\documentclass[12pt]{article}
\usepackage{amssymb, amscd, amsthm, amsmath,latexsym, amstext,mathrsfs,verbatim}
\usepackage[all]{xy}
\usepackage{graphicx,enumerate, url}
\def\lijntje{\vrule height2.4pt depth-2pt width0.5in}

\def\vlijntje{\vrule height0.45in depth0.4pt width0.4pt}
\def\vlijn{\buildrel {\hbox to 0pt{\hss$\textstyle\circ$\hss}}\over\vlijntje}

\def\dlijntje{{\vrule height2pt depth-1.6pt
width0.5in}\llap{\vrule height4pt depth-3.6pt width0.5in}}
\def\tlijntje{{\vrule height1.7pt depth-1.3pt
width0.5in}\llap{\vrule height3.0pt depth-2.6pt width0.5in}\llap{\vrule height4.3pt depth-3.9pt width0.5in}
}
\def\vtriple#1\over#2\over#3{\mathrel{\mathop{\kern0pt #2}\limits_{\hbox
to 0pt{\hss$#1$\hss}}^{\hbox to 0pt{\hss$#3$\hss}}}}
\def\rvtriple#1\over#2\over#3{\mathrel{\mathop{\kern0pt #2}\limits_{\hbox
to 0pt{\hss$#3$\hss}}^{\hbox to 0pt{\hss$#1$\hss}}}}
\def\Ct{\vtriple{\scriptstyle 2}\over\circ\over{}
\kern-1pt\lijntje\kern-1pt\vtriple{\scriptstyle 1}\over\circ\over{}
\kern-4pt{\dlijntje \kern -25pt<}\kern8pt
\vtriple{\scriptstyle 0}\over\circ\over{}\kern-1pt
}
\def\Bt{\vtriple{\scriptstyle 2}\over\circ\over{}
\kern-1pt\lijntje\kern-1pt\vtriple{\scriptstyle 1}\over\circ\over{}
\kern-4pt{\dlijntje \kern -25pt>}\kern8pt
\vtriple{\scriptstyle 0}\over\circ\over{}\kern-1pt}

\def\alg{{A}}

\def\ddA{{\rm A}}
\def\ddD{{\rm D}}

\def\Br{{\rm Br}}

\def\BrM{{\rm BrM}}

\def\ADE{{\rm ADE}}

\def\ddB{{\rm B}}

\def\ddD{{\rm D}}
\def\ddE{{\rm E}}

\def\hE{{\hat E}}

\def\Char {{\rm Char}}

\newcommand{\cA}{\mathcal{A}}

\newcommand{\C}{\mathbb C}

\newcommand{\Q}{\mathbb Q}
\newcommand{\R}{\mathbb R}
\newcommand{\Z}{\mathbb Z}

\newcommand{\fB}{\mathfrak{B}}

\def\Dm{\vtriple{\scriptstyle n+1}\over\circ\over{}\kern-1pt\lijntje\kern-1pt
\vtriple{\scriptstyle{n}}\over\circ\over{}
\cdots\cdots\vtriple{\scriptstyle 4}\over\circ\over{}\kern-1pt\lijntje\kern-1pt
\vtriple{\scriptstyle 3}\over\circ\over{\buildrel
{\scriptstyle 2}\over\vlijn}\kern-1pt\lijntje\kern-1pt
\vtriple{1}\over\circ\over{}\kern-1pt}

\def\Dn{\vtriple{\scriptstyle n}\over\circ\over{}\kern-1pt\lijntje\kern-1pt
\vtriple{\scriptstyle{n-1}}\over\circ\over{}
\cdots\cdots\vtriple{\scriptstyle 4}\over\circ\over{}\kern-1pt\lijntje\kern-1pt
\vtriple{\scriptstyle 3}\over\circ\over{\buildrel
{\scriptstyle 2}\over\vlijn}\kern-1pt\lijntje\kern-1pt
\vtriple{1}\over\circ\over{}\kern-1pt}
\def\En{\vtriple{\scriptstyle n}\over\circ\over{}\kern-1pt\lijntje\kern-1pt
\vtriple{\scriptstyle{n-1}}\over\circ\over{}
\cdots\cdots\vtriple{\scriptstyle 5}\over\circ\over{}\kern-1pt\lijntje\kern-1pt
\vtriple{\scriptstyle 4}\over\circ\over{\buildrel
{\scriptstyle 2}\over\vlijn}\kern-1pt\lijntje\kern-1pt
\vtriple{\scriptstyle 3}\over\circ\over{}\kern-1pt\lijntje\kern-1pt
\vtriple{\scriptstyle 1}\over\circ\over{}\kern-1pt}
\def\An{\vtriple{\scriptstyle n}\over\circ\over{}\kern-1pt\lijntje\kern-1pt
\vtriple{\scriptstyle{n-1}}\over\circ\over{}\kern-1pt\lijntje\kern-1pt
\vtriple{\scriptstyle n-2}\over\circ\over{}
\cdots\cdots
\vtriple{\scriptstyle 2}\over\circ\over{}\kern-1pt\lijntje\kern-1pt
\vtriple{\scriptstyle 1}\over\circ\over{}\kern-1pt}
\def\Cn{\vtriple{\scriptstyle n-1}\over\circ\over{}
\kern-1pt\lijntje\kern-1pt\vtriple{\scriptstyle{n-2}}\over\circ\over{}
\cdots\cdots
\vtriple{\scriptstyle 2}\over\circ\over{}
\kern-1pt\lijntje\kern-1pt\vtriple{\scriptstyle 1}\over\circ\over{}
\kern-4pt{\dlijntje \kern -25pt<}\kern10pt
\vtriple{\scriptstyle 0}\over\circ\over{}\kern-1pt}
\def\Ct{\vtriple{\scriptstyle 2}\over\circ\over{}
\kern-1pt\lijntje\kern-1pt\vtriple{\scriptstyle 1}\over\circ\over{}
\kern-4pt{\dlijntje \kern -25pt<}\kern12pt
\vtriple{\scriptstyle 0}\over\circ\over{}\kern-1pt
}
\def\Bn{\vtriple{\scriptstyle n-1}\over\circ\over{}
\kern-1pt\lijntje\kern-1pt\vtriple{\scriptstyle{n-2}}\over\circ\over{}
\cdots\cdots
\vtriple{\scriptstyle 2}\over\circ\over{}
\kern-1pt\lijntje\kern-1pt\vtriple{\scriptstyle 1}\over\circ\over{}
\kern-4pt{\dlijntje \kern -25pt>}\kern10pt
\vtriple{\scriptstyle 0}\over\circ\over{}\kern-1pt}
\def\Bt{\vtriple{\scriptstyle 2}\over\circ\over{}
\kern-1pt\lijntje\kern-1pt\vtriple{\scriptstyle 1}\over\circ\over{}
\kern-4pt{\dlijntje \kern -25pt>}\kern12pt
\vtriple{\scriptstyle 0}\over\circ\over{}\kern-1pt}
\def\Es{\vtriple{\scriptstyle 6}\over\circ\over{}\kern-1pt\lijntje\kern-1pt
\vtriple{\scriptstyle 5}\over\circ\over{}\kern-1pt\lijntje\kern-1pt
\vtriple{\scriptstyle 4}\over\circ\over{\buildrel
{\scriptstyle 2}\over\vlijn}\kern-1pt\lijntje\kern-1pt
\vtriple{3}\over\circ\over{}\kern-1pt\lijntje\kern-1pt
\vtriple{\scriptstyle 1}\over\circ\over{}\kern-1pt}
\def\Ff{
\vtriple{\scriptstyle 1}\over\circ\over{}
\kern-1pt\lijntje\kern-1pt\vtriple{\scriptstyle 2}\over\circ\over{}
\kern-4pt{\dlijntje \kern -25pt<}\kern10pt
\vtriple{\scriptstyle 3}\over\circ\over{}\kern-1pt\lijntje\kern-1pt
\vtriple{\scriptstyle 4}\over\circ\over{}
\kern-1pt}
\def\Ht{
\vtriple{\scriptstyle 1}\over\circ\over{}
\kern-1pt\overset{5}{\lijntje}\kern-1pt\vtriple{\scriptstyle 2}\over\circ\over{}
\kern-1pt\lijntje\kern-1pt
\vtriple{\scriptstyle 3}\over\circ\over{}\kern-1pt}
\def\Hf{
\vtriple{\scriptstyle 1}\over\circ\over{}
\kern-1pt\overset{5}{\lijntje}\kern-1pt\vtriple{\scriptstyle 2}\over\circ\over{}
\kern-1pt\lijntje\kern-1pt
\vtriple{\scriptstyle 3}\over\circ\over{}\kern-1pt\lijntje\kern-1pt
\vtriple{\scriptstyle 4}\over\circ\over{}
\kern-1pt}
\def\In{
\vtriple{\scriptstyle 0}\over\circ\over{}
\kern-1pt\overset{n}{\lijntje}\kern-1pt\vtriple{\scriptstyle 1}\over\circ\over{}
\kern-1pt}
\def\Gt{
\vtriple{\scriptstyle 0}\over\circ\over{}
\kern-4pt{\tlijntje\kern -25pt<}\kern 10pt\vtriple{\scriptstyle 1}\over\circ\over{}
\kern-1pt}
\def\EBn{\vtriple{\scriptstyle n-1}\over\circ\over{}
\kern-1pt\lijntje\kern-1pt\vtriple{\scriptstyle{n-2}}\over\circ\over{\buildrel
{\scriptstyle -1}\over\vlijn}\cdots\cdots
\vtriple{\scriptstyle 2}\over\circ\over{}
\kern-1pt\lijntje\kern-1pt\vtriple{\scriptstyle 1}\over\circ\over{}
\kern-4pt{\dlijntje \kern -25pt<}\kern8pt
\vtriple{\scriptstyle 0}\over\circ\over{}\kern-1pt}
\def\Cn{\vtriple{\scriptstyle n-1}\over\circ\over{}
\kern-1pt\lijntje\kern-1pt\vtriple{\scriptstyle{n-2}}\over\circ\over{}
\cdots\cdots
\vtriple{\scriptstyle 2}\over\circ\over{}
\kern-1pt\lijntje\kern-1pt\vtriple{\scriptstyle 1}\over\circ\over{}
\kern-4pt{\dlijntje \kern -25pt<}\kern10pt
\vtriple{\scriptstyle 0}\over\circ\over{}\kern-1pt}
\def\ECn{\vtriple{\scriptstyle -2}\over\circ\over{}
\kern-4pt{\dlijntje \kern -25pt>}\kern8pt\vtriple{\scriptstyle n-1}\over\circ\over{}
\kern-1pt\lijntje\kern-1pt\vtriple{\scriptstyle{n-2}}\over\circ\over{}
\cdots\cdots
\vtriple{\scriptstyle 2}\over\circ\over{}
\kern-1pt\lijntje\kern-1pt\vtriple{\scriptstyle 1}\over\circ\over{}
\kern-4pt{\dlijntje \kern -25pt<}\kern12pt
\vtriple{\scriptstyle 0}\over\circ\over{}\kern-1pt}
\def\Fo{\vtriple{\scriptstyle -1}\over\circ\over{}
\kern-1pt\lijntje\kern-1pt
\vtriple{\scriptstyle 1}\over\circ\over{}
\kern-1pt\lijntje\kern-1pt\vtriple{\scriptstyle 2}\over\circ\over{}
\kern-4pt{\dlijntje \kern -25pt<}\kern8pt
\vtriple{\scriptstyle 3}\over\circ\over{}\kern-1pt\lijntje\kern-1pt
\vtriple{\scriptstyle 4}\over\circ\over{}
\kern-1pt}
\def\Ft{
\vtriple{\scriptstyle 1}\over\circ\over{}
\kern-1pt\lijntje\kern-1pt\vtriple{\scriptstyle 2}\over\circ\over{}
\kern-4pt{\dlijntje \kern -25pt<}\kern8pt
\vtriple{\scriptstyle 3}\over\circ\over{}\kern-1pt\lijntje\kern-1pt
\vtriple{\scriptstyle 4}\over\circ\over{}
\kern-1pt\lijntje\kern-1pt
\vtriple{\scriptstyle -2}\over\circ\over{}
\kern-1pt}
\def\Go{\vtriple{\scriptstyle -1}\over\circ\over{}
\kern-1pt\lijntje\kern-1pt
\vtriple{\scriptstyle 0}\over\circ\over{}
\kern-4pt{\tlijntje\kern -25pt<}\kern 12pt\vtriple{\scriptstyle 1}\over\circ\over{}
\kern-1pt}
\def\Gf{
\vtriple{\scriptstyle 0}\over\circ\over{}
\kern-4pt{\tlijntje\kern -25pt<}\kern 12pt\vtriple{\scriptstyle 1}\over\circ\over{}
\kern-1pt\lijntje\kern-1pt
\vtriple{\scriptstyle -2}\over\circ\over{}
\kern-1pt}

\numberwithin{equation}{section}

\newtheorem{lemma}{Lemma}[section]

\newtheorem{prop}[lemma]{Proposition}
\newtheorem{thm}[lemma]{Theorem}

\theoremstyle{definition}

\newtheorem{defn}[lemma]{Definition}

\theoremstyle{remark}
\newtheorem{rem}[lemma]{Remark}

\newtheorem{example}[lemma]{Example}

\begin{document}
\title{Morita equivalences  on Brauer algebras  and BMW algebras of simply-laced types}
\author{ Shoumin Liu\footnote{The author is funded by the NSFC (Grant No. 11601275, Youth Program).}}
\date{}
\maketitle
%\mainmatter
%\setcounter{section}{-1} \tableofcontents

\begin{abstract}
The Morita equivalences of  classical Brauer algebras and classical  Birman-Murakami-Wenzl algebras have been  well studied.
Here we study the Morita equivalence problems on these two kinds of algebras of simply-laced type, especially  for them with the generic parameters.
We  show that Brauer algebras and  Birman-Murakami-Wenzl algebras of simply-laced type are Morita equivalent to
the direct sums of some group algebras of Coxeter groups and some  Hecke algebras of some Coxeter groups, respectively.
\end{abstract}

\section{Introduction}
In \cite{Brauer1937}, when the author study the invariant theory of orthogonal groups, the Brauer algebras are defined
as a class of  diagram algebras, which becomes the most classical examples in Schur-weyl duality. If we regard some horizontal strands
in the diagram algebras as roots of Coxeter groups of type $\ddA$, it is natural to define the Brauer algebras of other types associated to
other Dynkin diagrams. Cohen, Frenks, and  Wales define the Brauer algebras of simply-laced types in \cite{CFW2008}, and describe
some properties of these algebras. The Birman-Murakami-Wenzl algebras ($BMW$ in short) which is defined in \cite{BW1989} and \cite{M1987},
can be considered as a quantum version of
classical Brauer algebras. Analogously, the $BMW$  algebras can be extended to other simply-laced types  in \cite{CGW2008}. \\
The Morita equivalences (\cite{M1958})  and Quasi-heredity (\cite{CPS1988})  are important properties of associative algebras.
As  cellular algebras(\cite{Gra}, \cite{GL1996}, \cite{KX1999}),
these properties of classical Brauer algebras and $BMW$ algebras, even some related algebras,  are well studied  in many papers,
such as  K\"onig and Xi (\cite{KX1999}, \cite{KX19992},\cite{KX2001},\cite{X2000}),
Rui and Si(\cite{R2005}, \cite{RS2006}, \cite{RS2009}, \cite{RS2012}, \cite{S2015}). Their results are based on studying the bilinear forms
for  defining their cellular structures.\\
Therefore, it is natural to ask the Morita equivalences and quasi-heredity on the Brauer algebras and $BMW$ algebras of simply-laced types, especial
for type $\ddD_n$ and type $\ddE_n$ ($n=6, 7, 8$). Our paper will focus on these algebras with generic parameters, and is sketched as the following.
In Section \ref{sect:CA}, we first recall two equivalent definitions of cellular algebra from \cite{GL1996} and \cite{KX1999},
and introduce some basic properties of cellular algebras, especially about Morita equivalence.
In section \ref{sect:BA}, we recall the definition of Brauer algebras ($\Br(Q,k)$) of simply-laced types and some results from
\cite{CFW2008}. In section \ref{sect:ME}, we prove the Morita equivalence and quasi-heredity of $\Br(Q,k)$ with  some
conditions on ground field $k$ and generic parameter $\delta$. In section \ref{sect:typeD}, by analyzing the structure of $\Br(\ddD_n,k)$,
we show some results about the semi-simplicity of $\Br(\ddD_n,k)$ with $\delta$ evaluated. In section \ref{sect:BMW}, similar to Section \ref{sect:ME}, we present
 the Morita equivalence and quasi-heredity on $BMW$ algebras of simply-laced types.

\section{Cellular algebra}\label{sect:CA}
We first recall the definition of cellular algebra from \cite{Gra} and \cite{GL1996}.
\begin{defn}\label{GLdefn}
 An associative
algebra $\alg$ over a commutative ring $R$ is cellular if there is a quadruple
$(\Lambda, T, C, *)$ satisfying the following three conditions.

\begin{itemize}
\item[(C1)] $\Lambda$ is a finite partially ordered set.  Associated to each
$\lambda \in \Lambda$, there is a finite set $T(\lambda)$.  Also, $C$ is an
injective map
$$ \coprod_{\lambda\in \Lambda} T(\lambda)\times T(\lambda) \rightarrow \alg$$
whose image is an $R$-basis of $\alg$.

\item[(C2)]
The map $*:\alg\rightarrow \alg$ is an
$R$-linear anti-involution such that
$C(x,y)^*=C(y,x)$ whenever $x,y\in
T(\lambda)$ for some $\lambda\in \Lambda$.

\item[(C3)] If $\lambda \in \Lambda$ and $x,y\in T(\lambda)$, then, for any
element $a\in \alg$,
$$aC(x,y) \equiv \sum_{u\in T(\lambda)} r_a(u,x)C(u,y) \
\ \ {\rm mod} \ \alg_{<\lambda},$$ where $r_a(u,x)\in R$ is independent of $y$
and where $\alg_{<\lambda}$ is the $R$-submodule of $\alg$ spanned by $\{
C(x',y')\mid x',y'\in T(\mu)\mbox{ for } \mu <\lambda\}$.
\end{itemize}
Such a quadruple $(\Lambda, T, C, *)$ is called a {\em cell datum} for
$\alg$.
\end{defn}
There is also an equivalent definition due to K\"onig and Xi in \cite{KX1999}.
\begin{defn}\label{defn:KXcelluar}
 Let $A$ be $R$-algebra. Assume there is an anti-automorphism $i$ on $A$ with $i^2=id$. A two sided ideal $J$ in
$A$ is called cellular if and only if $i(J)=J$ and there exists a left ideal $\Delta\subset J$ such that $\Delta$ has finite rank and
there is an isomorphism of $A$-bimodules $\alpha:J\simeq \Delta\otimes_R i(\Delta)$ making the following diagram commutative:
\begin{center}
\xymatrix{
J\ar[rr]^{\alpha}\ar[d]_{i} &   & \Delta\otimes_R i(\Delta) \ar[d]^{x\otimes y\rightarrow i(y)\otimes i(x)} \\
          J \ar[rr]^{\alpha}          & & \Delta\otimes_R i(\Delta)
          }
\end{center}
The algebra $A$ is called \emph{cellular} if  there is a vector space decomposition $A=J'_1\oplus \cdots\oplus J'_n$ with
$i(J'_j)=J'_j$ for each $j$ and such that setting $J_j=\oplus_{k=1}^jJ'_j$ gives a chain of two sided ideals of
$A$ such that for each $j$ the quotient $J'_j=J_j/J_{j-1}$  is a cellular ideal of $A/J_{j-1}$.
\end{defn}
Also recall definitions of iterated inflations from \cite{KX1999}.
 Given an $R$-algebra $B$, a finitely generated free $R$-module $V$, and a bilinear form
 $\varphi: V\otimes_{R}V\longrightarrow  B$ with values in $B$, we define an  associative algebra (possibly without unit)
 $A(B, V, \varphi)$ as follows: as an $R$-module, $A(B, V, \varphi)$ equals $V\otimes_{R}V\otimes_{R}B$. The multiplication is defined on basis
 element as follows:
 \begin{eqnarray*}
 (a\otimes b \otimes x)(c\otimes d \otimes y):=a\otimes d \otimes x\varphi(b,c)y.
\end{eqnarray*}
Assume that there is an involution $i$ on $B$. Assume, moreover, that $i(\varphi(v,w))=\varphi(w,v)$.
If we can extend this involution $i$ to $A(B, V, \varphi)$ by defining $i(a\otimes b \otimes x)=b\otimes a \otimes i(x)$. Then
We call $A(B, V, \varphi)$ is an \emph{inflation} of $B$ along $V$.
Let $B$ be an inflated algebra (possible without unit) and $C$ be an algebra with unit. We define an algebra structure in such a way  that
$B$ is a two-sided ideal and $A/B=C$. We require that $B$ is an ideal, the multiplication is associative, and
that there exists a unit element of $A$ which maps onto the unit of the quotient $C$. The necessary conditions are outlined in
\cite[Section 3]{KX1999}. Then we call $A$  an inflation of $C$ along $B$, or iterated inflation of $C$ along $B$.
We present  Proposition 3.5 of \cite{KX1999}  below.
\begin{prop} An inflation of a cellular algebra is cellular again. In particular, an iterated inflation of $n$ copies of
$R$ is cellular, with a cell chain of length $n$ as in Definition \ref{defn:KXcelluar}.
\end{prop}
More precisely, the second statement has the following meaning. Start with $C$ a full matrix ring over $R$ and  $B$ an inflation of $R$ along a
free $R$-module, and form a new $A$ which is an inflation of the old $A$ along the new $B$, and continue this operation. Then after
$n$ steps we have produced a cellular algebra $A$ with a cell chain of length $n$.

We also have Theorem 4.1 from \cite{KX1999} as follows.
\begin{thm}Any cellular algebra over $R$ is the iterated inflation of finitely many copies of $R$. Conversely, any iterated inflation of
finitely many copies of $R$ is cellular.
\end{thm}
Let $A$ be  cellular(with identity) which can be realized as an iterated inflation of cellular algebras $B_l$ along vector spaces $V_l$
for $l=1, \ldots,n.$ This implies that as a vector space
\begin{eqnarray*}
A=\oplus_{l=1}^{n} V_l\otimes V_l\otimes B_l,
 \end{eqnarray*}
 and $A$ is cellular with a chain of two sided ideals ${0}=J_0\subset J_1\cdots \subset J_n=A$, which can be refined to a cell chain, and each
 quotient $J_l/J_{l-1}$ equals $V_l\otimes V_l\otimes B_l$ as an algebra without unit. The involution $i$ of $A$,is defined through the involution $i_l$ of the algebra $B_l$ where
 $i(a\otimes b\otimes x)=b\otimes a\otimes j_l(x)$. The multiplication rule of a layer $V_l\oplus V_l\oplus B_l$ is indicated by
 \begin{eqnarray*}
 (a\otimes b \otimes x)(c\otimes d \otimes y):=a\otimes d \otimes x\varphi(b,c)y+\text{lower terms}.
\end{eqnarray*}
 Here lower terms refers to element in lower layers $V_h\otimes V_h\otimes B_h$ for $h<l$.
 Let $1_{B_l}$ be the identity of the algebra $B_l$.\\
 We recall \cite[Theorem 2.6]{S2015} about the Morita equivalence of celluar algebra.
 \begin{thm}\label{thm:si}
  Let $R$ be a field. Suppose that $A$ is an iterated inflation of $R$-algebras $B_1$, $B_2$, $\cdots$, $B_n$, where each inflation
 is along $R$-vector space $V_i$, $1\leq i\leq n$. For each $i$, let $\varphi_i: V_i\otimes V_i\rightarrow B_i$ be the bilinear form with respect to
 each inflation. If $\varphi_i$ is non-singular for all $i$, then $$A\overset{Morita}{\sim}\oplus_{i=1}^{n}B_i.$$
 \end{thm}
 In this paper, we will focus on the quasi-heredity on some algebras, then  we recall the definition of quasi-heredity algebra from \cite{CPS1988}.
 \begin{defn} Let $k$ be any  associative ring, and  $A$ be a $k$-algebra. An ideal $J$ in $A$ is called a hereditary  ideal if $J$ is idempotent, $J(rad(A))J=0$. and J is a projective left(or, right) $A$-module; the algebra $A$ is called a heredity algebra.
 The algebra $A$ is called quasi-hereditary provided there is a finite chain $0=J_0\subset J_1\subset\dots \subset J_n=A$ of ideals in $A$ such that
 $J_j/J_{j-1}$ is a hereditary ideal in $A/J_{j-1}$ for all $j$. Such a chain is then called a heredity ideal of the quasi-hereditary algebra $A$.
\end{defn}
\section{Brauer algebras of simply-laced type}\label{sect:BA}
We recall the definition of simply-laced Brauer algebra from \cite{CFW2008}.

\begin{defn}\label{1.1}
Let $Q$ be a graph. The Brauer monoid $\BrM(Q)$ is the monoid
generated by the symbols $R_i$ and $E_i$, for  each node $i$ of $Q$ and $\delta$,
$\delta^{-1}$ subject to the following relation, where
$\sim$ denotes adjacency between nodes of $Q$.

\begin{equation}\delta\delta^{-1}=1     \label{1.1.1}
\end{equation}
\begin{equation}R_{i}^{2}=1          \label{1.1.2}
\end{equation}
\begin{equation}R_iE_i=E_iR_i=E_i     \label{1.1.3}
\end{equation}
\begin{equation}E_{i}^{2}=\delta E_{i}   \label{1.1.4}
\end{equation}
\begin{equation}R_iR_j=R_jR_i, \,\, \mbox{for}\, \it{i\nsim j} \label{1.1.5}
\end{equation}
\begin{equation}E_iR_j=R_jE_i,\,\, \mbox{for}\, \it{i\nsim j}  \label{1.1.6}
\end{equation}
\begin{equation}E_iE_j=E_jE_i,\,\, \mbox{for}\, \it{i\nsim j}    \label{1.1.7}
\end{equation}
\begin{equation}R_iR_jR_i=R_jR_iR_j, \,\, \mbox{for}\, \it{i\sim j}  \label{1.1.8}
\end{equation}
\begin{equation}R_jR_iE_j=E_iE_j ,\,\, \mbox{for}\, \it{i\sim j}       \label{1.1.9}
\end{equation}
\begin{equation}R_iE_jR_i=R_jE_iR_j ,\,\, \mbox{for}\, \it{i\sim j}     \label{1.1.10}
\end{equation}
The Brauer algebra $\Br(Q)$ is  the the free $\Z$-algebra for Brauer monoid  $\BrM(Q)$.
 \end{defn}
 We denote $\Br(Q,k)=\Br(Q)\otimes_{\Z} k$, where $k$ is an arbitrary ring.
 The Brauer algebras $\Br(Q)$ has been well studied in \cite{CFW2008}, where the basis and ranks
 of finite types are given. Usually we call $R_i$s Coxeter generators, and  $E_i$s  Temperley-Lieb generators(\cite{TL1971}).
% Let $M$ be a connected double laced or simply laced Dynkin diagram of  finite type, namely type $\ddA_n$, $\ddB_n$, $\ddC_n$,
%$\ddD_n$, $\ddE_n$($n=6$, $7$, $8$), $\ddF_4$.
%We  list their Dynkin  diagrams in
% Table \ref{DKdiagram}.
 \begin{table}[!htb]
\caption{Coxeter diagrams of spherical types}\label{DKdiagram}
\begin{center}
\begin{tabular}{c|c}
type&diagram\\
\hline
$\ddA_n$&$\An$\\
$\ddD_n$&$\Dn$\\
$\ddE_n$, $6\le n\le 8$&$\En$
\end{tabular}
\end{center}
\end{table}

Let $Q$ be a spherical Coxeter diagram of simply laced type, i.e., its connected components are of type
$\ddA$, $\ddD$, $\ddE$ as listed in Table \ref{DKdiagram}. This section is to summarize some results in \cite{CGW2006}.

When $Q$ is  $\ddA_n$, $\ddD_n$, $\ddE_6$, $\ddE_7$, or $\ddE_8$, we denote it as $Q\in{\rm ADE}$.
Let $(W, T)$ be the Coxeter system of type $Q$ with $T=\{R_1,\ldots,R_n\}$ associated to the diagram of
$Q$ in Table \ref{DKdiagram}.
Let $\Phi$ be the root system of type $Q$,  let $\Phi^+$ be its positive root system, and let $\alpha_i$ be the simple root
associated to the node $i$ of $Q$. We are interested in sets $B$ of mutually commuting reflections, which has a bijective correspondence with sets of
mutually orthogonal roots of $\Phi^+$, since each reflection in $W$ is uniquely determined by a positive root and vice versa.
\begin{rem}\label{rem:positiveaction}
The action of $w\in W$ on $B$ is given by conjugation in case $B$ is described by reflections and given by
$w\{\beta_1,\ldots, \beta_p\}=\Phi^+\cap \{\pm w\beta_1,\ldots, \pm w\beta_p \}$, in case $B$ is described by positive roots.
For example, $R_4R_1R_2R_1\{\alpha_1+\alpha_2, \alpha_4\}=\{\alpha_1+\alpha_2,\alpha_4\}$, where $Q=\ddA_4$.
\end{rem}
For $\alpha$, $\beta\in \Phi$, we write $\alpha\sim\beta$ to denote $|(\alpha,\beta)|=1$. Thus, for $i$ and $j$ nodes of
$Q$, we have $\alpha_i\sim\alpha_j$ if and only if $i\sim j$.
\begin{defn}
Let $\mathfrak{B}$ be a   $W$-orbit of sets of  mutually orthogonal positive roots. We say that
$\mathfrak{B}$ is an \emph{admissible orbit} if for each $B\in \mathfrak{B}$, and $i$, $j\in Q$ with $i\not\sim j$
 and $\gamma$, $\gamma-\alpha_i+\alpha_j\in B$ we have $r_iB=r_jB$, and each element in $\mathfrak{B}$ is called an admissible root set.
\end{defn}
This is the definition from \cite{CGW2006},  and there is another equivalent definition in \cite{CFW2008}. We also state it here.
\begin{defn}
 Let $B\subset\Phi^+$ be a mutually orthogonal root set. If for all $\gamma_1$, $\gamma_2$, $\gamma_3\in B$
 and $\gamma\in \Phi^+$, with $(\gamma,\gamma_i)=1$, for
 $i=1$, $2$, $3$, we have $2\gamma+\gamma_1+\gamma_2+\gamma_3\in B$, then $B$ is called an admissible root set.
 \end{defn}
 By these two definitions, it follows that the intersection of two admissible root sets are admissible.
 It can be checked by definition that the intersection of two admissible sets are still admissible.  Hence
for a given  set $X$ of mutually orthogonal positive roots, the unique smallest admissible set containing
$X$ is called the admissible closure of $X$, and denoted as $X^{\rm cl}$ (or $\overline{X}$). Up to
the action of the corresponding Weyl groups,
all admissible root sets of type $\ddA_n$, $\ddD_n$,  $\ddE_6$, $\ddE_7$, $\ddE_8$ have appeared in  \cite{CFW2008}, \cite{CGW2009} and \cite{CW2011},
and are listed
in  Table \ref{table:admADE}. In the table, the
set
$Y(t)^*$  consists of all $\alpha^*$ for  $\alpha\in Y(t)$, where $\alpha^*$ is    the unique
positive root orthogonal to $\alpha$ and all other positive roots orthogonal
to $\alpha$ for type $\ddD_n$ with $n> 4$.  For type $\ddD_n$, if we considier the root systems are realized in $\R^{n}$,
with $\alpha_1=\epsilon_2-\epsilon_1$, $\alpha_2=\epsilon_2+\epsilon_1$, $\alpha_i=\epsilon_i-\epsilon_{i-1}$, for $3\le i\le n$,
then $\Phi^+=\{\epsilon_j\pm\epsilon_i\}_{1\le i<j\le n}$, then  $(\epsilon_j\pm\epsilon_i)^*=\epsilon_j\mp\epsilon_i$.  For $\ddD_4$,
the $t$ can be $0$, $1$, $2$, $3$, which means the number of nods in the  coclique.
When $t=2$, although in the Dynkin diagram $\{\alpha_1,\alpha_2\}$ and $\{\alpha_1,\alpha_4\}$ are symmetric,  they are  in the
 different orbits under the  Weyl group's actions.  Then the admissible root sets for $\ddD_4$ can be written as the $W(\ddD_4)$'s
 orbits of $\emptyset$, $\{\alpha_3\}$,  $\{\alpha_1,\alpha_2\}$,  $\{\alpha_1,\alpha_4\}$, and $\{\alpha_1,\alpha_2,\alpha_4,\alpha_1+\alpha_2+\alpha_4+2\alpha_3\}.$
\begin{table}
\caption{Admissible root sets of simply laced type}
\label{table:admADE}
 \begin{center}
\begin{tabular}{|c|c|}
\hline
$Q$&representatives of orbits \, under\, $W(Q)$\\
\hline
$\ddA_n$&$\{\alpha_{2i-1}\}_{i=1}^{t}$,\,$0\le t\le \left\lfloor{(n+1)/2}\right\rfloor. $\\
\hline
$\ddD_n$&$Y(t)=\{\alpha_{n+2-2i}, \alpha_{n-2},\ldots, \alpha_{n+2-2t}\}$ \, $0\le t\le \left\lfloor{n/2}\right\rfloor. $\\
        & $\{\alpha_{n+2-2i}, \alpha_{n-2},\ldots, \alpha_{4}, \alpha_1\}$  \text{if}    $2|n$\\
        & $Y(t)\cup Y(t)^*$ \, $0\le t\le \left\lfloor{n/2}\right\rfloor$\\
\hline
$\ddE_6$ & $\emptyset$, $\{\alpha_6\}$, $\{\alpha_6, \alpha_4\}$,  $\{\alpha_6, \alpha_2, \alpha_3\}^{\rm cl}$ \\
\hline
$\ddE_7$ & $\emptyset$, $\{\alpha_7\}$, $\{\alpha_7, \alpha_5\}$,  $\{\alpha_5, \alpha_5, \alpha_2\}$, $\{\alpha_7, \alpha_2, \alpha_3\}^{\rm cl}$,
                               $\{\alpha_7, \alpha_5, \alpha_2, \alpha_3\}^{\rm cl}$   \\
\hline
$\ddE_8$ & $\emptyset$, $\{\alpha_8\}$, $\{\alpha_8, \alpha_6\}$,   $\{\alpha_8, \alpha_2, \alpha_3\}^{\rm cl}$,
                               $\{\alpha_8, \alpha_5, \alpha_2, \alpha_3\}^{\rm cl}$   \\
\hline
\end{tabular}
\end{center}
\end{table}

\begin{example} If $Q=\ddD_4$, the root set $\{\alpha_1, \alpha_2, \alpha_4\}$ is mutually orthogonal but not admissible,
and its admissible closure is $\{\alpha_1, \alpha_2, \alpha_4, \alpha_1+\alpha_2+2\alpha_3+\alpha_4\}$.
\end{example}
\begin{defn}
\label{df:cA}
Let $\cA$ denote the collection of all admissible subsets of $\Phi$ consisting of
mutually orthogonal  positive roots.
Members of $\cA$ are called admissible sets.
%Suppose that if $X\subset \Phi$ consists of mutually orthogonal roots, then the minimal
%admissible containing $X$ as a subset is called the \emph{admissible closure} of $X$, denoted as $X^{\rm cl}$.
\end{defn}
Now we consider the actions of  $R_i$ on an admissible $W$-orbit $\fB$. When $R_iB\neq B$, We say that
$R_i$ lowers $B$ if there is a root $\beta\in B$ of minimal height  among those moved by $R_i$ that
satisfies $\beta-\alpha_i\in \Phi^+$ or $R_iB<B$.
We say that $R_i$ raises $B$ if there is a root $\beta\in B$ of minimal height among
those moved by $R_i$ that satisfies $\beta+\alpha_i\in \Phi^+$ or $R_iB>B$. By this
we can set an  partial order on $\fB=WB$. The poset $(\fB, <)$ with this minimal ordering is called the monoidal poset (with respect to $W$) on $\fB$ (so $\fB$ should be admissible for the poset to be monoidal). If $\fB$ just consists of sets
of a single root, the order is determined by the canonical height function on roots. There is an important conclusion in \cite{CGW2006},  stated  below.
This theorem plays a crucial role in obtaining a basis for Brauer algebra of simply laced type in \cite{CFW2008}.
\begin{thm}\label{thm:maximal}
 There is a unique maximal element in $\fB$.
 \end{thm}

 For any $\beta\in\Phi^+$ and $i\in\{1,\ldots,n\}$, there exists a $w\in
W$ such that $\beta = w\alpha_i$. Then $R_\beta := wR_iw^{-1}$ and
$E_\beta := wE_iw^{-1}$ are well defined (this is well known from Coxeter
group theory for $R_\beta$;
see \cite[Lemma 4.2]{CFW2008} for $E_\beta$).  If
$\beta,\gamma\in\Phi^+$ are mutually orthogonal, then $E_\beta$ and
$E_\gamma$ commute (see \cite[Lemma 4.3]{CFW2008}). Hence, for $B\in\cA$, we
 define the product
\begin{eqnarray}
\label{eqn:EprodB}
E_B &=& \prod_{\beta\in B} E_\beta,
\end{eqnarray}
which is a quasi-idempotent, and the normalized version
\begin{eqnarray}
\label{eqn:EhatB}
\hE_B &=& \delta^{-|B|} E_B,
\end{eqnarray}
which is an idempotent element of the Brauer monoid.
For a mutually orthogonal root subset $X\subset \Phi^+$, we have
\begin{eqnarray}
\label{eqn:Ecloure}
E_{X^{\rm cl}}=\delta^{|X^{\rm cl}\setminus X|}E_{X}.
\end{eqnarray}
% There exists $w\in W $ and a simple root
%$\alpha_{i}$ such that $\beta=w\alpha_{i}$, for any $\beta\in \Psi^+$. We define the element
%$E_{\beta}$ of $\Br(\ddA_{2n-1})$ by
%$$E_{\beta}=wE_{i}w^{-1}.$$
%By \cite{CFW2008}, we know that $E_{\beta}$ is well defined.
%
%For an arbitrary admissible mutually orthogonal positive  root set $X\subset\Phi^{+}$, we define the element $E_{X}$ by
%$$E_{X}=\prod_{\beta\in X}E_{\beta}.$$ Note that this is well defined as the
%factors in this product commute (by \cite{CFW2008}).
Let $C_X=\{i\in Q\mid \alpha_i \perp X\}$ and
let $W(C_{X})$ be the subgroup generated by the
generators of nodes in $C_X$.
The subgroup $W(C_{X})$ is called the \emph{centralizer} of $X$.
The normalizer of $X$, denoted by $N_{X}$ can be defined as
$$N_{X} =\{w\in W\mid E_X w=w E_X\}.$$
We let $D_X$ denote a set of right coset representatives for $N_X$ in $W$.\\
In \cite[Definition 3.2]{CFW2008}, an action of the Brauer monoid
$\BrM(Q)$ on the collection $\cA$ of admissible root sets
in $\Phi^+$ was indicated below, where  $Q\in {\rm ADE}$.
\begin{defn}\label{eq:aboveaction}
There is an action of the Brauer monoid $\BrM(Q)$ on the collection
$\cA$.  The generators $R_i$ $(i=1,\ldots,n)$ act by the natural action of
Coxeter group elements on its  positive root sets as in Remark \ref{rem:positiveaction},
% where negative roots are negated so
%as to obtain positive roots,
and the element $\delta$ acts as the identity,
and the action of $E_i$ $(i=1,\ldots,n)$ is defined by
\begin{equation}
E_i B :=\begin{cases}
B & \text{if}\ \alpha_i\in B, \\
(B\cup \{\alpha_{i}\})^{\rm cl} & \text{if}\ \alpha_i\perp B,\\
R_\beta R_i B & \text{if}\ \beta\in B\setminus \alpha_{i}^{\perp}.
\end{cases}
\end{equation}
\end{defn}

We will refer to this action as the admissible set action.
This monoid action plays an important role in getting a basis of $\BrM(Q)$ in \cite{CFW2008}. For the basis, we state one  conclusion from \cite[Proposition 4.9]{CFW2008} below.
 \begin{prop}\label{ADErewform}
  Each element of the Brauer monoid $\BrM(Q)$ can be written in the form $$\delta^k uE_{X}zv,$$ where
 $X$ is the highest element from one $W$-orbit in $\cA$, $u$, $v^{-1}\in D_X$, $z\in W(C_X)$, and $k\in \Z$.
 \end{prop}
\section{Morita equivalence on $\Br(Q,k)$ with generic parameter $\delta$}\label{sect:ME}
Here we suppose that $k$ is a field.
To satisfy the cellular condition of corresponding Hecke algebra  in \cite{G2007} and
good prime property in \cite{L1984}, and also \cite[Table 3]{CFW2008}, we suppose the
following for the characteristic of $k$($\Char(k)$).
\begin{equation}\label{chark}
\begin{cases}
\Char(k) \quad\text{no condition},\quad \text{when}\quad  Q=\ddA_n,\\
\Char(k)\neq 2,\quad \text{when}\quad  Q=\ddD_n,\\
\Char(k)\neq 2,3, \quad \text{when}\quad  Q=\ddE_6,\ddE_7,\\
\Char(k)\neq 2,3,5 \quad \text{when}\quad  Q=\ddE_8.
\end{cases}
\end{equation}
Recall the representation $\rho_{\fB}$ of $\BrM(Q)$ from \cite[Lemma 3.4]{CFW2008}.
Let $V_{\fB}$ be the free right $k[\delta^{\pm 1}][W(C_{\fB})]$ with basis
$\xi_{B}$ for $B\in \fB$, where $W(C_{\fB})=W(C_{X})$ with $X$ being the highest element of $\fB$.
we define $R_iV_{B}=V_{R_i B}h_{B,i},$
where $h_{B,i}$ is defined in \cite[Definition 2]{CGW2006}.
 The action of $E_i$ $(i=1,\ldots,n)$  on $V_{\fB}$ is defined by
\begin{equation}
E_i \xi_B :=\begin{cases}
\xi_B & \text{if}\ \alpha_i\in B, \\
0 & \text{if}\ \alpha_i\perp B,\\
R_\beta R_i \xi_B & \text{if}\ \beta\in B\setminus \alpha_{i}^{\perp}.
\end{cases}
\end{equation}
\begin{thm}\label{thm:semisimple}
Let $Q\in\ADE$.
\begin{enumerate}[(i)]
\item The  associative algebra $\Br(Q,k)$ is free over $k[\delta^{\pm 1}]$  and of rank as given in the \cite[Table 2]{CFW2008},
and the algebra is semisimple when tensored with $k(\delta)$.

\item
For each irreducible representation $\tau$ of $W(C_{\fB})$, we denote ${\rm dim}(\tau)$ is the dimension of the representation $\tau$.
The algebra $\Br(Q,k)\otimes_{\Z[\delta,\delta^{-1}]}k(\delta)$ is a direct sum of matrix algebras of size $|\fB|\cdot{\rm dim}(\tau)$ for
$(B,\tau)$ running over all pairs of $W$-orbits $\fB$ in $\cA$ and any irreducible representation $\tau$ of $W(C_{\fB})$.
\end{enumerate}
\end{thm}
\begin{proof} The first half of $(i)$ follows from \cite[Theorem 1.1]{CFW2008}. Checking \cite[Table 3]{CFW2008},
the number $\Char(k)$ is a good prime for every  $W(C_{\fB})$ with all $W$-orbits $\fB$ in $\cA$ for type $Q$, which implies that
$k(\delta)[W(C_{\fB})]$ is split semisimple.
Then the similar argument for proving \cite[Theorem 1.1]{CFW2008}
can be applied to the second half of $(i)$. The conclusion of $(ii)$ holds for  the similar arguments in \cite[Corollary 5.6]{CFW2008}
 and semisimplicity of $k(\delta)[W(C_{\fB})]$ of  all $W$-orbits $\fB$ in $\cA$ for type $Q$.
\end{proof}
\begin{thm} \label{thm:morita}
For the algebra $\Br(Q,k)\otimes_{\Z[\delta,\delta^{-1}]}k(\delta)$, we have
$\Br(Q,k)\otimes_{\Z[\delta,\delta^{-1}]}k(\delta)\overset{morita}{\sim}\oplus k(\delta)[W(C_{\fB})],$
where $\fB$ runs over all the $W$-orbits  in $\cA$,  and $W(C_{\emptyset})=W$ when $\fB=\emptyset$.
Furthermore, the  algebra $\Br(Q,k)\otimes_{\Z[\delta,\delta^{-1}]}k(\delta)$ is quasi-hereditary.
\end{thm}
\begin{proof}
From the proof of the cellularity theorem \cite[Theorem 1.2]{CFW2008} in \cite[Section 6]{CFW2008},
we have that  $\Br(Q, k)\otimes_{\Z[\delta,\delta^{-1}]}k(\delta)$ is an iterated inflation
of $k(\delta)[W(C_{\fB})]$, where $\fB$ runs over all the $W$-orbits  in $\cA$, including $\fB=\emptyset$ and $W(C_{\emptyset})=W.$
It follows that
 $$\Br(Q, k)\otimes_{\Z[\delta,\delta^{-1}]}k(\delta)\cong \bigoplus_{\fB} V_{\fB}\otimes V_{\fB} \otimes k(\delta)[W(C_{\fB})].$$
By Theorem \ref{thm:semisimple}, the algebra $\Br(Q, k)\otimes_{\Z[\delta,\delta^{-1}]}k(\delta)$ is semisimple, therefore
the bilinear forms for defining the iterated inflation structure of $\Br(Q, k)\otimes_{\Z[\delta,\delta^{-1}]}k(\delta)$
$$\varphi_{\fB}: V_{\fB}\otimes V_{\fB} \rightarrow k(\delta)[W(C_{\fB})]$$
 are non-singular by \cite[Theorem 3.8]{GL1996}. Hence the theorem holds for Theorem \ref{thm:si}.
 The  bilinear forms is non-singular, then the   algebra $\Br(Q,k)\otimes_{\Z[\delta,\delta^{-1}]}k(\delta)$ is quasi-hereditary follows from \cite[Remark 3.10]{GL1996}.
\end{proof}
\begin{rem} If we take  $\{v_{B}\}_{B\in \fB}$ as the basis of $V_{\fB}$, we see that $\varphi_{\fB}$ is a matrix over  $ k(\delta)[W(C_{\fB})]$ with coefficients of
polynomial of variable $\delta$. By Theorem \ref{thm:morita}, when we consider this algebra with evaluating $\delta$ in $k$, there are only finite
values of $\delta$ so that  $\varphi_{\fB}$ fails to be non-singular.  So there are only finite $\delta$s in  $k$ so that $\Br(Q, k)\otimes_{\Z[\delta,\delta^{-1}]}k(\delta)$
fails to be semisimple.
\end{rem}
\begin{rem} Theorem \ref{thm:morita} is a generalization of \cite[Theorem 7.3]{KX2001}, which says that the
classical Brauer algebra
$$\mathscr{B}_{n}(\delta)\overset{Morita}{\sim}\bigoplus_{i=1}^{\left\lfloor n/2\right\rfloor} k\mathfrak{S}_{n-2i},$$
where the classical Brauer algebra is considered as our Brauer algebra of type $\ddA_{n-1}$.
\end{rem}
\begin{rem} Let $A=\Br(Q,k)\otimes_{\Z[\delta,\delta^{-1}]}k(\delta)$. Since $A$  is quasi-hereditary, we have the following.
\begin{enumerate}[(i)]
\item The algebra $A$ has finite global dimension bounded by $2\sum_{\fB}W(C_{\fB})-2$,
where $\fB$ runs over all the $W$-orbits  in $\cA$, including $\fB=\emptyset$ and $W(C_{\emptyset})=W$(\cite{DR1989}).
\item The Cartan matrix of $A$ has determinant $1$(\cite[Theorem 3.1]{KX19992}).
\item Any cell chain of $A$ is a hereditary Chain (\cite[Theorem 3.1]{KX19992}).
\end{enumerate}
\end{rem}
\section{Semisimplicity of  $\Br(\ddD_n,k)$ with $\delta$ specialized}\label{sect:typeD}
Now we focus on type $\ddD_n$, and the parameter $\delta$ is  specialized in $k$.

Let $\Lambda$ be the cells used to prove the celluarity in \cite[Section 6]{CFW2008}.
The underlying set $\Lambda$ is defined as the union of $\Lambda_1$ and
$\Lambda_2$, where $\Lambda_1=\{t\}_{t=0}^{[\frac{n}{2}]}$
and $\Lambda_2=\{(t,\theta)\}_{t=1}^{[\frac{n+1}{2}]}$.
The partial
order on $\Lambda$ is given by
\begin{itemize}
\item $\lambda_1 >\lambda_2$ if
$\lambda_1=t_1<t_2=\lambda_2\in \Lambda_1$,
\item  $\lambda_1 >\lambda_2$ if
$\lambda_1=(t_1,\theta)$, $\lambda_2=(t_2,\theta)\in \Lambda_2$ and
$t_1<t_2$,
\item  $\lambda_1 >\lambda_2$ if $\lambda_1=t_1\in \Lambda_1$ and
$\lambda_2=(t_2,\theta)\in \Lambda_2$ and $t_1\leq t_2$.
\end{itemize}
\noindent
It can be
illustrated by the following Hasse diagram, where $a>b$ is equivalent to the
existence of a directed path from $a$ to $b$.
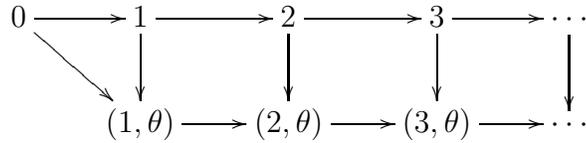
\begin{center}
\phantom{longgggg}
\quad
\xymatrix{
0\ar[dr]\ar[r]&1\ar[r]\ar[d] & 2 \ar[d]\ar[r] &3\ar[d]\ar[r]&\cdots\ar[d]\\
         & (1,\theta) \ar[r] & (2,\theta)\ar[r]&(3,\theta)\ar[r]&\cdots
          }
\end{center}
And we see that the order on $\Lambda$ is partially ordered,  but not totally ordered.

As we have seen in the Table \ref{table:admADE}, There is a class of $Y(t)\cup Y(t)^{*}$,
for $1\leq t\leq \frac{n}{2}$.
The $\Lambda_2$ is corresponding to the orbits of those $Y(t)\cup Y(t)^{*}$.
By the \cite[Section 1]{CFW2008} or by the diagram representation of Brauer algebra of type
$\ddD_n$ in \cite{CGW2009}, we see that that the structure  and the bilinear forms associated to these cells
are the same as $\Br(\ddA_{n-1},k)$, except the  $\delta$ in $\Br(\ddA_{n-1},k)$ is replaced by
$\delta^2$, because we replaced each generator of $E_i$ in $\Br(\ddA_{n-1},k)$  by  $E_iE_i^*$ in $\Br(\ddD_n,k)$.

We keep the notation from \cite{R2005} and \cite{RS2006}, Let
$$\Z(n)=\{i\in \Z \mid 4-2n\leq i\leq n-2\}\setminus \{i\in \Z \mid 4-2n\leq i\leq 3-n, 2\nmid i\}.$$
By \cite[Theorem 1.2, Theorem 1.3]{R2005} and our analysis about the structure of $\Br(\ddD_n,k)$, we have the followings.
\begin{thm}
Let $k=\C$. The algebra  $\Br(\ddD_n,k)$ is not semisimple, when
\begin{enumerate}[(i)]
\item the parameter $\delta^2\in \Z(n)$ and $\delta\neq 0$, or
\item the parameter $\delta=0$ and $n\notin \{1,3,5\}$.
\end{enumerate}
 \end{thm}
\begin{thm}
Let $k$ be  a field with  characteristic $e>0$. The algebra  $\Br(\ddD_n,k)$ is not semisimple, when
\begin{enumerate}[(i)]
\item the parameter $\delta^2\in \Z(n)$, $\delta\neq 0$ and $e\nmid n!$, or
\item the parameter $\delta=0$ and $n\notin \{1,3,5\}$  and $e\nmid n!$.
\end{enumerate}
 \end{thm}
 \section{The Morita equivalence on BMW algebras of simply laced types}\label{sect:BMW}
 The Birman-Murakami-Wenzl (BMW in short) algebras are first introduced in \cite{BW1989} and \cite{M1987},
 which can be considered  as type  $\ddA$ in \cite{CGW2005},
 where the authors extended them to all simply laced types. We present the definition in the below.
  \begin{defn}\label{BMWdefn}
 Let $Q$ be a simply laced Coxeter diagram of rank $n$. The Birman-Murakami-Wenzl algebra of type
$Q$ is the algebra, denoted by $\ddB(Q)$, with ground field $\Q(l,\delta)$, where $l$ and $\delta$ are
transcendental and algebraically independent over $\Q$,  whose presentation is given
on generators $g_i$ and $e_i$ ($i=1$,$2$,$\ldots$, $n$) by the following relations
\begin{eqnarray}
g_ig_j&=&g_jg_i  \qquad \quad  \qquad \quad \quad\mbox{for}\ i\nsim j \\
g_ig_jg_i&=&g_jg_ig_j  \qquad \quad \qquad\,\kern-.04em\quad \mbox{for}\ {i\sim j}\   \\
me_i&=&l(g_i^2+mg_i-1) \qquad \kern.05em \mbox{for}\,\mbox{any}\ i  \\
g_ie_i&=&l^{-1}e_i \qquad \quad  \qquad \kern.05em\quad\quad \mbox{for}\,\mbox{any}\ i \\
e_ig_je_i&=&le_i  \qquad\quad \quad \qquad\,\kern-.04em \quad\quad\mbox{for}\ {i\sim j}\
\end{eqnarray}
where $m=(l-l^{-1})/(1-\delta).$
%and  $i\sim j$ means that $i$ and $j$ are
%adjacent in the Dynkin diagram $Q$,
%and  $i\nsim j$ indicates that they are distinct and  non-adjacent.
\end{defn}
\begin{rem}
It is known there is a natural homomorphism of rings from the BMW algebra of type $Q$ to the Brauer algebra of the same type induced on the generators
by $g_i\mapsto R_i$ and $e_i\mapsto E_i$  with $l=1$. In \cite{CGW2005}, \cite{CGW2008}, \cite{CW2011}, it is proved that
the rewritten form in Proposition \ref{ADErewform} gives a basis of the BMW algebra by changing $R_i$ to $g_i$ and $E_i$ to  $e_i$.
\end{rem}
Similar to the Theorem \ref{thm:semisimple}, we have the following conclusion for the algebra $\ddB(Q)$.
\begin{thm}
Let $Q\in\ADE$.
\begin{enumerate}[(i)]
\item The  associative algebra $\ddB(Q)$ is semisimple.

\item
For the algebra $\ddB(Q)$, we have
$\ddB(Q)\overset{morita}{\sim}\oplus \mathcal{H}(C_{\fB}),$
where $\fB$ runs over all the $W$-orbits  in $\cA$, $\mathcal{H}(C_{\fB})$ is the Hecke algebra of type $C_{\fB}$.
Furthermore, the  algebra  $\ddB(Q)$ is quasi-hereditary.
\end{enumerate}
\end{thm}
\begin{proof}For (i), when $Q$ is of type $\ddD_{n}$, $4\leq n$, or $\ddE_n$, $n=6$, $7$, $8$, the proofs are given in \cite[Theorem 1.1]{CGW2008},
and \cite[Theorem 1]{CW2011}, respectively. For type $\ddA_n$, by modifying our parameters with the parameters in \cite{RS2009} and \cite{RS2012}, we can see for the generic
parameters, and  the algebra $\ddB(\ddA_n)$ is semisimple because of \cite[Theorem 4.3]{RS2009} and \cite[Theorem 2.17]{RS2012}.\\
For (ii), the prooof of the  case for $Q=\ddA_n$ can be found in \cite[Theorem 2.17]{RS2012}. For $Q$ being other types ,we apply the the argument of the
  proof of Theorem \ref{thm:morita}. By (i), the algebra $\ddB(Q)$ is semisimple, therefore
the bilinear forms for defining the iterated inflation structure of $\ddB(Q)$
$$\varphi_{\fB}^{'}: V_{\fB}\otimes V_{\fB} \rightarrow \mathcal{H}(C_{\fB})$$
which is defined in \cite[Section 8]{CGW2008} for type $\ddD_n$ and in the proof of  \cite[Theorem 8]{CW2011},
 are non-singular by \cite[Theorem 3.8]{GL1996}. Hence the Morita Equivalence in (ii) holds for Theorem \ref{thm:si}.
 The  bilinear forms is non-singular, then the   algebra $\ddB(Q)$ is quasi-hereditary follows from \cite[Remark 3.10]{GL1996}.
\end{proof}
\begin{rem}If we evaluate  the parameters of  $\Br(Q,k)$ and $\ddB(Q)$, we have to compute the Gram determinants of
the bilinear forms defining their cellular structures as in \cite{RS2009}, to judge the semi-simplicity, Morita equivalence and quasi-heredity
about them. Some further work is needed  to complete  parameter problems about this.
There are also Brauer algebras of multiply-laced type which has already defined and studied in \cite{CLY2010},\cite{CL2011}, and \cite{L2015},
and we can explore these properties about them.
\end{rem}

Shoumin Liu\\
Email: s.liu@sdu.edu.cn\\
School of Mathematics, Shandong University\\
Shanda Nanlu 27, Jinan, \\
Shandong Province, China\\
Postcode: 250100

\end{document}